\documentclass[10pt,reqno]{amsart}    

\usepackage{amssymb,amsmath,amscd}

\advance\voffset by -10truemm
\def\Com{\mathop {\fam0 Com}\nolimits}
\def\As {\mathop {\fam0 As }\nolimits}
\def\Lie{\mathop {\fam0 Lie}\nolimits}
\def\Var{\mathop {\fam0 Var}\nolimits}
\def\Nov{\mathop {\fam0 Nov}\nolimits}
\def\Der{\mathop {\fam0 Der}\nolimits}
\def\id{\mathop {\fam0 id}\nolimits}

\def\Hom{\mathop {\fam0 Hom}\nolimits}

\let\<\langle
\let\>\rangle

\newtheorem{lemma}{Lemma}

\newtheorem{theorem}{Theorem}

\theoremstyle{definition}

\newtheorem{remark}{Remark}
\newtheorem{example}{Example}

\begin{document}

\title{Derived identities of differential algebras}

\author{P. S. Kolesnikov}

\address{Sobolev Institute of Mathematics}
\email{pavelsk@math.nsc.ru}

\begin{abstract}
Suppose $A$ is a not necessarily associative algebra with a derivation~$d$.
Then $A$ may be considered as a system with two binary operations $\succ $ and 
$\prec $ defined by $x\succ y = d(x)y$, $x\prec y = xd(y)$, $x,y\in A$. 
Suppose $A$ satisfies some multi-linear polynomial identities. 
We show how to find the identities that hold for operations $\prec $ and $\succ $.
It turns out that if $A$ belongs to a variety governed by an operad $\Var $ 
then $\succ $ and $\prec $ satisfy the defining
relations of the operad $\Var\circ \Nov$, where $\circ $ is the Manin white product of operads, 
$\Nov $ is the operad of Novikov algebras. Moreover, there are no other independent identities that hold 
for operations $\succ $, $\prec $ on a differential $\Var $-algebra.
\end{abstract}

 
 \maketitle
 
\section{Introduction}

Let $A$ be an algebra over a field $\Bbbk $, i.e., a linear space equipped with a
binary linear operation (multiplication) $\cdot : A\otimes A \to A$. 
Suppose $T$ is a linear operator on $A$, and let $\prec$ and $ \succ $ be two 
new linear maps $A\otimes A\to A$ defined by
\[
a\prec b = a\cdot T(b), \quad a\succ b= T(a)\cdot b, \quad a,b\in A.
\]
Denote the system $(A,\prec, \succ)$ by $A^{(T)}$.
If the initial algebra $A$ satisfies a polynomial identity then what could be said about $A^{(T)}$? 
The answer is known if $T$ is a Rota---Baxter operator \cite{BGuoNi13}, \cite{GubKol13}, \cite{FardGuo07} or 
averaging operator \cite{GubKol14}. In these cases, the identities of $A^{(T)}$ may be obtained by 
means of categoric procedures (black and white Manin products of operads \cite{GK94}). 

The purpose of this note is to consider the case when $T$ is a derivation (or generalized derivation). 
It is well-known \cite{DzhLowf09}
that a commutative (and  associative) algebra $A$ with a derivation $d$ induces Novikov 
algebra structure on $A^{(d)}$, 
assuming $(a\succ b)=(b\prec a)$. Conversely, if an identity holds on $A^{(d)}$ for an arbitrary 
commutative algebra $A$ with a derivation $d$ then this identity is a consequence of Novikov 
identities. 

In this paper, we generalize this observation for an arbitrary variety $\Var $ of algebras. Namely, 
if an identity $f$ holds on $A^{(d)}$ for every $\Var $-algebra $A$ with a derivation $d$
then we say $f$ is a derived identity of $\Var $. For $\Var = \As$, some derived identities were 
found in \cite{Lod10}. We show that for a multi-linear variety $\Var $ the set of derived identities 
coincides with the set of relations on the operad $\Var\circ \Nov$, where 
$\Var $ and $\Nov $ are the operads governing the varieties $\Var $ and $\Nov$, respectively, 
$\circ $ is the Manin white product of operads.

Calculation of Manin products is a relatively simple linear algebra problem, it is based on finding 
intersections of vector spaces. Therefore, our result provides an easy way for finding a complete list 
of derived identities for an arbitrary binary operad.

\section{White product of operads}

Suppose $\Var $ is a multi-linear variety of algebras, i.e., 
a class of all algebras that satisfy a given family of multi-linear 
identities (over a field $\Bbbk $ of characteristic zero, every variety is multi-linear). 
Fix a countable set of variables $X=\{x_1,x_2,\dots \}$ and denote 
\[
 \Var(n) = M_n(X)/M_n(X)\cap T_{\Var}(X), 
\]
where $M_n(X)$ is the space of all (non-associative) multi-linear polynomials of degree $n$ 
in $x_1,\dots, x_n$, $T_{\Var}(X)$ is the T-ideal of all identities that hold in $\Var $. 

The collection of spaces $(\Var(n))_{n\ge 1}$ forms a symmetric operad relative to the natural 
composition rule and symmetric group action (see, e.g., \cite{LodVal08}). We will denote this operad by the same symbol $\Var $. 
Every algebra $A\in \Var $ may be considered as a morphism of multi-categories
$\Var \to \Hom_\Bbbk$, where $\Hom_\Bbbk $ stands for the multi-category of linear spaces
over $\Bbbk $ and multi-linear maps. 
The operad $\Var $ is a {\em binary} one, i.e., it is generated (as a symmetric operad) 
by the elements of $\Var(2)$. 

\begin{example}\label{exmp:Lie}
The operad $\Lie $ governing the variety of Lie algebras 
is generated by 1-dimensional space $\Lie(2)=\Bbbk\mu $, 
$\mu^{(12)}=-\mu$. The space $\Lie(1)$ is also 1-dimensional, it is spanned by 
the identity $\id $ of the operad.
If we identify $\id $ with $x_1$ and $\mu $ with $[x_1x_2]$ then 
\[
 \mu(\id, \mu) = [x_1[x_2x_3]], \quad \mu(\mu,\id ) = [[x_1x_2]x_3],  
\]
so the Jacobi identity may be expressed as 
\[
 \mu(\id,\mu) - \mu(\id,\mu)^{(12)} = \mu(\mu, \id). 
\]
\end{example}
 
\begin{example}\label{exmp:As} 
 The operad $\As $ governing the variety of associative algebras 
 is generated by 2-dimensional space $\As(2)$ spanned by $\nu=(x_1x_2)$
 and $\nu^{(12)}=(x_2x_1)$. Associativity relations form an $S_3$-submodule 
 in $M_3(X)$ spanned by 
 \[
  \nu(\nu,\id) = \nu(\id,\nu).
 \]
\end{example}

\begin{example}\label{exmp:Nov}
{\em Novikov algebra} is a linear space with a multiplication
satisfying the following axioms:
\begin{gather}
 (x_1x_2)x_3 - x_1(x_2x_3) = (x_2x_1)x_3 - x_2(x_1x_3)   \label{eq:LSymm},\\
 (x_1x_2)x_3 = (x_1x_3)x_2. \label{eq:RComm}
\end{gather} 
The corresponding operad is generated by 2-dimensional $\Nov(2)$ 
spanned by $\nu = (x_1x_2)$ and $\nu^{(12)}=(x_2x_1)$. Defining 
identities of the variety $\Nov $ may be expressed as 
\[
\begin{gathered}
\nu(\nu, \id) - \nu(\id ,\nu) = \nu(\nu^{(12)}, \id ) - \nu^{(12)}(\nu, \id)^{(23)}, \\
\nu(\nu, \id) = \nu(\nu, \id)^{(23)}.
\end{gathered}
\]
\end{example}

Let $\Var_1$ and $\Var_2$ be two  operads. Then the family of spaces
$(\Var_1(n)\otimes \Var_2(n))_{n\ge 1}$ is an operad relative to 
the natural (componentwise) composition and symmetric group action
(known as the {\em Hadamard product} of $\Var_1$ and $\Var_2$). 
Even if $\Var_1$ and $\Var_2$ were binary operads, their Hadamard product
$\Var_1\otimes \Var_2$
 may  be non-binary. The sub-operad of $\Var_1\otimes \Var_2$ generated by 
 $\Var_1(2)\otimes \Var_2(2)$ is known as {\em Manin white product} of $\Var_1$ and $\Var_2$, 
 it is denoted $\Var_1\circ \Var_2$ \cite{GK94}. 
 
\begin{example}\label{exmp:Lie-Nov}
 The operad $\Lie \circ \Nov$ is isomorphic to the operad governing the class of all 
 algebras ({\em magmatic} operad).
\end{example}

Indeed, both $\Lie $ and $\Nov $ are quadratic operad, and so is $\Lie\circ \Nov$ \cite{GK94}. 
Let $\mu $ and $\nu $ be the generators of $\Lie$ and $\Nov$. 
It is enough to find the defining identities of $\Lie \circ \Nov$ that are quadratic with 
respect to $\mu $ and $\nu $. 

Identify $\mu\otimes \nu $ with $[x_1\prec x_2]$, then $\mu\otimes \nu^{(12)} = -(\mu\otimes \nu)^{(12)}$
corresponds to $-[x_2\prec x_1]$. Hence, $(\Lie\circ \Nov)(n)$ is an image of the space $M_n(X)$ 
of all multi-linear non-associative polynomials of degree $n$ in $X=\{x_1,x_2,\dots \}$ relative to the 
operation $[\cdot \prec \cdot ]$. Calculating the  compositions
$\mu(\mu,\id)\otimes \nu(\nu,\id)$ and $\mu(\id,\mu)\otimes \nu(\id,\nu)$, we obtain
\[
\begin{gathered}
m_1= [[x_1\prec x_2]\prec x_3] = [[x_1x_2]x_3]\otimes (x_1x_2)x_3, \\
m_2= [x_1\prec [x_2\prec x_3]] = [x_1[x_2x_3]]\otimes x_1(x_2x_3), 
\end{gathered} 
\]
It remains to find the intersection of the $S_3$-submodule generated by $m_1$ and $m_2$ 
in $M_3(X)\otimes M_3(X)$ with the kernel of the projection 
$M_3(X)\otimes M_3(X)\to \Lie(3)\otimes \Nov(3)$. Straightforward calculation
shows the intersection is zero. Hence, the operation $[\cdot \prec \cdot]$
satisfies no identities.

\begin{example}\label{exmp:As-Nov}
 The operad $\As \circ \Nov$ generated by 4-dimensional space $(\As\circ \Nov)(2)$
 spanned by $(x_1\succ x_2)$, $(x_2\succ x_1)$, $(x_1\prec x_2)$, $(x_2\prec x_1)$
 relative to the following identities:
 \begin{gather}
  (x_1\succ x_2)\prec x_3 - x_1\succ (x_2\prec x_3)=0, \label{eq:L-alg} \\
  (x_1\prec x_2)\prec x_3 -x_1\prec (x_2\succ x_3)+ (x_1\prec x_2)\succ x_3 - x_1\succ (x_2\succ x_3)=0.
								  \label{eq:LodDer}
 \end{gather}
\end{example}

This result may be checked with a straightforward computation, as in Example \ref{exmp:Lie-Nov}. 
Namely, one has to find the intersection of the $S_3$-submodule in $M_3(X)\otimes M_3(X)$
generated by 
\[
\begin{gathered}
 (x_1\prec x_2)\prec x_3 = (x_1x_2)x_3\otimes (x_1x_2)x_3, \\
 x_1\succ (x_2\succ x_3) = x_1(x_2x_3)\otimes (x_3x_2)x_1, \\
 (x_1\succ x_2)\prec x_3 = (x_1x_2)x_3\otimes (x_2x_1)x_3, \\
 x_1\succ (x_2\prec x_3) = x_1(x_2x_3)\otimes (x_2x_3)x_1, \\
 (x_1\prec x_2)\succ x_3 = (x_1x_2)x_3 \otimes x_3(x_1x_2), \\
 x_1\prec (x_2\succ x_3) = x_1(x_2x_3)\otimes x_1(x_3x_2), \\
 x_1\prec (x_2\prec x_3) = x_1(x_2x_3)\otimes x_1(x_2x_3), \\ 
 (x_1\succ x_2)\succ x_3 = (x_1x_2)x_3\otimes x_3(x_2x_1)
\end{gathered}
\]
with the kernel of the projection $M_3(X)\otimes M_3(X)\to \As(3)\otimes \Nov(3)$.

\section{Derived identities}

Given a (non-associative) algebra $A$ with a derivation $d$, denote by 
$A^{(d)}$ the same linear space $A$ considered as a system with two binary 
linear operations of multiplication  $(\cdot \prec \cdot)$ 
and $(\cdot \succ \cdot)$ defined by
\[
 x\prec y = xd(y), \quad x\succ y = d(x)y,\quad x,y\in A.
\]

Let $\Var $ be a multi-linear variety of algebras. As above, denote by $T_{\Var}(X)$ 
the T-ideal of all identities in a set of variables $X$ that hold on $\Var $. 
A non-associative polynomial $f(x_1,\dots, x_n)$ in two operations
of multiplication $\prec $ and $\succ $ is called 
a {\em derived identity} of $\Var $ if for every $A\in \Var $ and for every derivation 
$d\in \Der (A)$ the algebra $A^{(d)}$ satisfies the identity $f(x_1,\dots ,x_n)=0$.
Obviously, the set of all derived identities is a T-ideal of the algebra of non-associative 
polynomials in two operations.

For example,  $(x\succ y)=(y\prec x)$ is a derived identity of $\Com $. 
Moreover, the operation $(\cdot \prec \cdot )$ satisfies
the axioms of Novikov algebras \eqref{eq:LSymm} and \eqref{eq:RComm}. 
It was actually shown in \cite{DzhLowf09} that the entire T-ideal of derived 
identities of $\Com $ is generated by these identities. 

For the variety of associative algebras, it was mentioned in 
\cite{Lod10} that 
\eqref{eq:L-alg} and \eqref{eq:LodDer} are derived identities of~$\As $.

\begin{remark}[c.f. \cite{GuoKeig08}]
For a multi-linear variety $\Var$, the
free differential $\Var$-algebra generated by a set $X$ 
with one derivation is nothing but 
$\Var\<d^\omega X\>$, where $d^\omega X = \{d^s(x) \mid s\ge 0, x\in X\}$.
\end{remark}

Indeed, consider the free magmatic algebra $F(X;d)=M\<d^\omega X\>$ and define a linear map 
$d: F(X;d)\to F(X;d)$ in such a way that $d(d^s(x))=d^{s+1}(x)$, $d(uv)=d(u)v+ud(v)$.
Since $\Var $ is defined by multi-linear relations, the T-ideal $T_{\Var} (d^\omega X)$ 
is $d$-invariant. Hence, $U=F(X;d)/T_{\Var}(d^\omega X)$ is a $\Var $-algebra with a derivation~$d$.
Since all relations from $T_{\Var}(d^\omega X)$ hold in every differential $\Var$-algebra, 
$U$ is free. 

Let $\mathcal N_{\Var }$ be the class of differential $\Var $-algebras with one locally nilpotent 
derivation, i.e., for every $A\in \mathcal N_{\Var}$ and for every $a\in A$ there 
exists $n\ge 1$ such that $d^n(a)=0$. 

The following statement shows why the variety generated by $\mathcal N_{\Var }$ coincides 
with the class of all differential $\Var $-algebras. 

\begin{lemma}\label{lem:LocNilp}
Suppose $f=f(x_1,\dots, x_n)$ is a multi-linear identity that 
holds on $A^{(d)}$ for all $A\in \mathcal N_{\Var }$. 
Then $f$ is a derived identity of $\Var $.
\end{lemma}

\begin{proof}
Consider the free differential $\Var$-algebra $U_n$ generated 
by the set $X=\{x_1,x_2,\dots \}$
with one derivation $d$ modulo defining relations $d^{n}(x)=0$, $x\in X$. 

Then $U_n\in \mathcal N_{\Var }$. As a differential $\Var$-algebra, $U_n$ is a homomorphic 
image of the free magma $M\<d^{\omega }X\>$. 
Denote by $I_n$ the kernel of the homomorphism $M\<d^{\omega }X\> \to U_n$.
As an ideal of $M\<d^{\omega }X\>$, $I_n$ is a sum of two ideals: $I_n=T_{\Var}(d^\omega X)+N_n$,
where $N$ is generated by $d^{n+t}(x)$, $x\in X$, $t\ge 0$. 
Note that the last relations form a $d$-invariant subset of $M\<d^\omega X\>$, 
so the ideal $N_n$ is $d$-invariant. 

If a polynomial $F\in M\<d^\omega X\>$ belongs to $I_n$ and 
its degree in $d$ is less than $n$ then $F$ belongs to $T_{\Var }(d^\omega X)$. 
If $f=f(x_1,\dots , x_n)$ is a multi-linear identity in $U_n^{(d)}$ 
then the image $F$ of $f$ in $M\<d^\omega X\>$ has degree $n-1$ in $d$, 
so $F\in T_{\Var }(d^\omega )$. 
Hence, $f$ is a derived identity of $\Var $. 
\end{proof}

\begin{theorem}\label{thm:WhiteNov->Derived}
If a multi-linear identity $f$ holds on the variety 
governed by the operad $\Var\circ \Nov $ then 
$f$ is a derived identity of $\Var $. 
\end{theorem}

\begin{proof}
For every $A\in \mathcal N_{\Var }$ we may construct an algebra in the variety 
$\Var \circ \Nov $ as follows. Consider the linear space $H$ spanned by elements $x^{(n)}$, $n\ge 0$, with 
multiplication 
\[
 x^{(n)}\cdot x^{(m)} = \binom{n+m-1}{n} x^{(n+m-1)}.
\]
This is a Novikov algebra (in characteristic 0, this is just the ordinary polynomial algebra).
Denote $\hat A = A\otimes \Bbbk [x]$ and define operations $\prec $ and $\succ $ on $\hat A$ by
\[
 (a\otimes f)\prec (b\otimes g) = ab\otimes f\cdot g, 
 \quad 
 (a\otimes f)\succ (b\otimes g) = ab\otimes g\cdot f.
\]
The algebra $\hat A$ obtained belongs to the variety governed by the operad $\Var \circ \Nov $.
There is an injective map 
\begin{align}
 & A \to \hat A, \nonumber \\
 & a\mapsto \sum\limits_{s\ge 0} d^s(a)\otimes x^{(s)}, \label{eq:A-to-hatA}
\end{align}
which preserves operations $\prec $ and $\succ $. Indeed, apply \eqref{eq:A-to-hatA} to 
$a\prec b$, $a,b\in A$:
\begin{multline}\nonumber
(a\prec b)=ad(b) \mapsto \sum\limits_{s\ge 0}  d^s(ad(b))\otimes x^{(s)} 
= \sum\limits_{s,t\ge 0} \binom{s+t}{s} d^s(a)d^{t+1}(b)\otimes x^{(s+t)} \\
= \sum\limits_{s,t\ge 0}  d^s(a)d^{t+1}(b) \otimes x^{(s)}\cdot x^{(t+1)} 
=
\left(\sum\limits_{s\ge 0} d^s(a)\otimes x^{(s)}\right)
\prec
\left(\sum\limits_{t\ge 0} d^t(a)\otimes x^{(t)}\right).
\end{multline}
Therefore, $A^{(d)}$ is a subalgebra of a $(\Var\circ \Nov )$-algebra, 
so all 
defining identities of $\Var\circ \Nov $ hold on $A^{(d)}$.
Lemma \ref{lem:LocNilp} completes the proof.
\end{proof}

\begin{theorem}\label{thm:Converse}
Let $f$ be a  multi-linear derived identity of $\Var $. Then $f$ holds on all ($\Var\circ\Nov$)-algebras.
\end{theorem}

\begin{proof}
Let $X=\{x_1,x_2,\dots \}$, and let $f(x_1,\dots , x_n)$ be a non-associative multi-linear polynomial 
in two operations of multiplication $\prec$ and $ \succ$. 
By definition, $f$ is an identity of the algebra $\Var\<X\>\otimes \Nov\<X\>$ 
with operations 
\[
(u\otimes f)\prec (v\otimes g)= (uv\otimes fg), 
\quad 
(u\otimes f)\succ (v\otimes g)= (uv\otimes gf)
\]
if and only if $f$ is a defining identity of $\Var\circ \Nov$.

It is well-known that $\Nov\<X\>$ is a subalgebra of $C^{(d)}$ for 
the free differential commutative algebra $C$ generated by $X$ with a derivation~$d$
\cite{DzhLowf09}. Therefore, 
$\Var\<X\>\otimes \Nov\<X\>$ is a subalgebra of $(\Var\<X\>\otimes C)^{(d)}$,
where $d(u\otimes f)=u\otimes d(f)$ for $u\in \Var\<X\>$, $f\in C$.
Hence, $\Var\<X\>\otimes \Nov\<X\>$ is embedded into a differential 
$\Var $-algebra.
\end{proof}

\begin{remark}
In this paper, we consider algebras with one binary operation, so $\Var $ has only one generator. 
However, Theorems~\ref{thm:WhiteNov->Derived}, \ref{thm:Converse} are easy to generalize for algebras 
with multiple binary operations, i.e., for an arbitrary binary operad $\Var $. 
\end{remark}

\begin{remark}
In the definition of a derived identity, a derivation $d$ of an algebra $A$ 
may be replaced with {\em generalized derivation}, 
i.e., a linear map $D: A\to A$ such that 
\[
D(ab)=D(a)b+aD(b)+\lambda ab, \quad a,b\in A,
\]
where $\lambda $ is a fixed scalar in $\Bbbk $. Indeed, it is enough to note that
 $D-\lambda \id$ is an ordinary derivation.
\end{remark}

\end{document}